\documentclass[11pt,a4paper]{article}
\usepackage[utf8x]{inputenc}
\usepackage{ucs}
\usepackage[top=1in, bottom=1in, left=1in, right=1in]{geometry}
\usepackage{amsmath}
\usepackage{hyperref}
\usepackage{amsfonts,dsfont}
\usepackage{amssymb}
\usepackage{amsthm}
\usepackage{graphicx}
\usepackage[nottoc, notlof, notlot]{tocbibind}  
\usepackage{enumitem}  
\usepackage{setspace}
\author{Xiaolin ZENG}
\date{}
\title{How Vertex reinforced jump process arises naturally}
\newtheorem{lem}{Lemma}
\newtheorem{thm}{Theorem}
\newtheorem{defi}{Definition}
\newtheorem{prop}{Proposition}
\newtheorem{rmk}{Remarks}
\def\P{\mathbb{ P}}
\def\E{\mathbb{ E}}

\begin{document}
\maketitle

\begin{abstract}
  We prove that the only nearest neighbor jump process with local dependence on the occupation times satisfying the partially exchangeable property is the vertex reinforced jump process, under some technical conditions (Theorem~\ref{thm1}). This result gives a counterpart to the characterization of edge reinforced random walk given by Rolles~\cite{rolles2003edge}.
\end{abstract}
\section{Introduction}
One of the most remarkable results in probabilistic symmetries is the de Finetti's theorem~\cite{de1931funzione}, which states that the law of any exchangeable sequence valued in a finite state space is in fact a mixture of i.i.d.\ sequences. This theorem has a geometrical interpretation via Choquet's theorem. More precisely, the subspace of exchangeable probabilities forms a convex, and those probabilities given by i.i.d.\ sequences are exactly the extreme points of the convex~\cite{aldous1985exchangeability}.

In the 1920s, W.E. Johnson~\cite{zabell1982we} conjectured that, under some technical conditions, if a process \(X_n\) is exchangeable and \(\P(X_{n+1}=i\vert X_0,\ldots,X_n)\) depends only on the number of times \(i\) occurs and the total steps \(n\), then \(X_n\) is nothing but the famous Polya urn: drawing balls uniformly from an urn and put back one additional ball with same color as the drawn one. This is a process with linear reinforcement. In term of random walk, the natural counterpart of Polya urn is the edge reinforced random walk (ERRW). Diaconis conjectured that this process have the same characterization as Polya urn. In~\cite{rolles2003edge} S.W.W.Rolles have shown that both conjectures are true under technical conditions.

The vertex reinforced jump process (VRJP) is a linearly reinforced process in continuous time. In a recent paper, Sabot and Tarres~\cite{sabot2011edge} have shown that ERRW is a mixture of VRJP, which indicates that the VRJP are building blocks of ERRW, thus should share a similar characterization. This paper gives this characterization (Theorem~\ref{thm1}), as a counterpart of Rolles' result; namely, the only continuous time process which is partially exchangeable and the transition probability depends only on neighbor local times is VRJP, under technical conditions.

\vspace{0.4cm}

Let us first recall the definition of ERRW, let \(G=(V,E)\) be a locally finite undirected graph without direct loops (edges with one endpoint). Let \(Z_n\) denote the location of the random process at time \(n\). Let \(a_e>0\), \(e\in E\). For \(n\in \mathbb{N}\), define \(w_n(e)\), the weight of edge \(e\) at time \(n\), by
\begin{align*}
&w_0(e)=a_e\ \text{for all \(e\in E\) },\\
&w_{n+1}(e)=
\begin{cases}
  w_n(e)+1& \text{for }e=\{Z_n,Z_{n+1}\}\in E,\\
  w_n(e)& \text{for }e\in E\setminus\{\{Z_n,Z_{n+1}\}\}.
\end{cases}\end{align*}
Let \(\P_{v_0}^{(a)}\) denote the probability of the ERRW on \(G\) starting at \(v_0\) with initial weights $a=(a_{e})_{e\in E}$. Then \(\P_{v_0}^{(a)}\) is defined by
\begin{align*}
&Z_0=v_0, \ \P_{v_0}^{(a)}-a.s.,\\
&\P_{v_0}^{(a)}(Z_{n+1}=v|Z_0,\ldots,Z_n)=
\begin{cases}
  \frac{w_n(\{Z_n,v\})}{\sum_{e,Z_n\in e}w_n(e)}& \text{if }\{Z_n,v\}\in E\\
  0 &\text{otherwise.}
\end{cases}\end{align*}

Now let us introduce some definitions before stating Rolles' result.
Again \(G=(V,E)\) is a locally finite undirected graph without direct loops, with its vertex set \(V\) and edge set \(E\). Denote \(i\sim j\) if \(\{i,j\}\in E\). Following Rolles, we call \((Z_n)_{n\geq 0}\) a nearest neighbor random walk on \(G\), if it is a discrete time random process (not necessarily Markov) such that successive positions are neighbors.

An admissible path of the random walk is a sequence of vertices of \(G\), denoted \(\pi=(v_0,v_1,\ldots,v_n)\) such that consecutive vertices are neighbors. The number of visits to  vertex \(i\) of  path \(\pi\) is denoted
\[N_i(\pi):=\#\{k:\ v_k=i,\ k=0,\ldots, n\};\]Similarly, the number of transition counts in the path \(\pi\) of an oriented edge \(e=(i,j)\) is denoted
\[N_e(\pi)=N_{i,j}(\pi):=\#\{k:\ v_k=i,v_{k+1}=j,\ k=0,\ldots,n-1\}.\]
Two paths \(\xi,\eta\) are said to be equivalent and denoted \(\xi \sim \eta\), if \(\xi\) and \(\eta\) start at the same state and the transition counts from \(i\) to \(j\) of any pair \((i,j)\) are equal for \(\xi\) and \(\eta\), i.e. \(N_{i,j}(\xi)=N_{i,j}(\eta)\) for all \((i,j)\).
\begin{rmk}
  Two equivalent paths necessarily end at the same vertex.
\end{rmk}

\begin{defi}
  A nearest neighbor random walk is partially exchangeable if any two equivalent paths have the same probability.
\end{defi}

\begin{thm}[Diaconis \& Freedman~\cite{diaconis1980finetti}]
 Let \(Z_n\) be a recurrent random walk (i.e.\ with probability one it returns to \(Z_0\) infinitely often), then \(Z\) is a mixture of Markov chains if and only if it is partially exchangeable. Moreover, the mixing measure is uniquely determined.
\end{thm}

As it turns out that edge reinforced random walk is a mixture of reversible Markov chains, Rolles introduced the following more restrictive notion of partial exchangeability: for \(\pi=(v_0, \ldots,v_n)\) and \(e=(i,j)\) let
\[\tilde{N}_e(\pi):=\#\{k:\ v_k=i,v_{k+1}=j \text{ or }v_k=j,v_{k+1}=i,\ k=0,\ldots,n-1\}.\]
\begin{defi}
\label{revPE}
  A nearest neighbor random walk is partially exchangeable in a reversible sense if it satisfies the following: for any two paths \(\xi,\eta\), if \(\tilde{N}_e(\xi)=\tilde{N}_e(\eta)\) for all \(e\in E\), then \(\xi\) and \(\eta\) have the same probability.
\end{defi}
In~\cite{rolles2003edge} Theorem 1.1, Rolles proved that if a nearest neighbor random walk is recurrent and partially exchangeable in a reversible sense, then it is a mixture of reversible Markov chain.

\vspace{0.7cm}
\noindent Rolles' main result in~\cite{rolles2003edge} states that, if \(G=(V,E)\) is a strongly connected graph and \(Z_n\) is a nearest neighbor random walk on \(G\) such that the following assumptions are satisfied:
  \begin{enumerate}
  \item \(Z\) is partially exchangeable in a reversible sense (Definition~\ref{revPE}).
  \item  For all \(v\in V\) and \(e\in E\) there exists a function \(f_{v,e}\) taking values in \([0,1]\) such that for all \(n\geq 0\)
\[\P(Z_{n+1}=v|\mathcal{F}_n)=f_{Z_n,e}(N_{Z_n}(Z_0,\ldots,Z_n),\tilde{N}_{Z_n,v}(Z_0,\ldots,Z_n)).\]
  \end{enumerate}
Then \(Z\) is an edge reinforced random walk or a Markov chain under some technical conditions (c.f.~\cite{rolles2003edge} for precision).

\vspace{0.4cm}

Next we define the vertex reinforced jump process \(X_t\). Assign positive weights \((W_e)_{e\in E}\) to the edges, the process \(X_t\) starts at time \(0\) at some vertex \(i_0\), if \(X\) is at vertex \(i\in V\) at time \(t\), then, conditioned on the past, the process jumps to a neighbor \(j\) of \(i\) with rate \(W_{i,j}(1+l_j(t))\),
 where for \(e=\{i,j\}\), \(W_{i,j}=W_e\) and \(l_j(t)\) is the local time of vertex \(j\) at time \(t\):
\[l_j(t):=\int_0^t \mathds{1}_{X_s=j}ds.\]

\begin{thm}[Sabot \& Tarres\cite{sabot2011edge}]
   The ERRW \({Z}_n\) with weights \((a_e)\) is equal in law to the discrete time process associated with a VRJP \(X_t\) in random independent weights \(W_e\sim \operatorname{Gamma}(a_e,1)\)
\end{thm}
\noindent And finally, the VRJP \(X_t\) turns out to be partially exchangeable within a time scale (c.f.\ next section for the definition of partial exchangeability in continuous times). Let \[D(s)=\sum_{i\in V}(l_i(s)^2+2l_i(s)),\] then the process \(Y_t=X_{D^{-1}(t)}\) is a mixture of Markov processes with an explicit mixing measure, in addition, the mixing measure turns out to be related to a \(\sigma\)-model introduced by Zirnbauer, c.f.~\cite{sabot2011edge} Theorem 2.

In this paper we give a counterpart of Rolles' result for VRJP, namely we characterize exchangeable jump processes with local rate functions.
\section{Definitions and results}

\begin{defi}
\label{defi1}
We call \((X_t)_{t\geq 0}\)  a nearest neighbor jump process on \(G\), if it is a random process which is right continuous without explosion, and each jump is from some vertex \(i\) to one of its neighbors \(j\) (i.e. \(i\sim j\)).
\end{defi}

\begin{defi}
  A nearest neighbor jump process \(X_t\) is a  mixture of Markov jump processes if there exists a probability measure \(\mu\) on Markov jump processes such that \(\mathcal{L}(X_t)=\int \mathcal{L}(Y_t) \mu(dY)\), where \(\mathcal{L}\) denotes the law of respective processes. If for \(\mu\) a.s.\ the Markov processes are reversible, then the process \(X_t\) is a mixture of reversible Markov processes.
\end{defi}

\noindent Freedman introduced the notion of partial exchangeability in continuous time in~\cite{freedman1996finetti}.

\begin{defi}[Freedman]
\label{defFreedman}
  A continuous process \(X_t\) is partially exchangeable if for each \(h>0\), the law of \(\{ X_{nh}; n=1,2,\cdots\}\) satisfies the following property:
for any two paths \(\xi=(\xi_0,\ldots,\xi_l),\ \eta=(\eta_0,\ldots,\eta_l)\) such that  \( \xi\sim \eta\ \),
\[ \P(X_0=\xi_0,\ldots,X_{lh}=\xi_l)=\P(X_0=\eta_0,\ldots,X_{lh}=\eta_l).\]
\end{defi}
\noindent We recall the de Finetti's theorem in continuous time showed by Freedman~\cite{freedman1996finetti}.
\begin{thm}
  Let \(X_t\) be a continuous time process starting at \(i_0\in G\), \(X_t\) is mixture of Markov jump processes if
  \begin{enumerate}
  \item \(X_t\) has no fixed points of discontinuity, more precisely, for every \(t\), if \(t_n\rightarrow t\), then \(\P(X_{t_n}\rightarrow X_t)=1\);
  \item \(X_t\) is recurrent;
  \item \(X_t\) is partially exchangeable.
  \end{enumerate}
\end{thm}

\noindent Our main theorem is:
\begin{thm}
\label{thm1}
Let \(X_t\) be a nearest neighbor jump process on \(G\) satisfying the following assumptions:
\begin{enumerate}
\item For all \(i\in V\), there exists \(\mathcal{C}^2\) diffeomorphisms \(h_i\) such that \(X\) is partially exchangeable within the time scale \(D(s)=\sum_{i\in V}h_i(l_i(s))\);
\item \(G\) is strongly connected (i.e.\ any two adjacent vertices are in a cycle);
\item The process, at vertex \(i\) at time \(t\), jumps to a neighbor \(j\) of \(i\) with rate \(f_{i,j}(l_j(t))\) for some continuous functions \(f_{i,j}\)
\end{enumerate}
Then \(X\) is a vertex reinforced jump process within time scale, i.e.\ there exists another time scale \(\tilde{D}\) such that \(X_{\tilde{D}^{-1}(t)}\) is a vertex reinforced jump process.
\end{thm}

\begin{rmk}
In fact, the hypothesis of Theorem~\ref{thm1} implies that the functions \(f_{i,j}(x)\) are necessarily of the form \(W_{i,j}x+\varphi_j\).
\end{rmk}
 \begin{rmk}
Note that we do not a priori require \(f_{i,j}=f_{j,i}\), i.e.\ there is no assumption of reversibility for \(X_t\); however the VRJP is a mixture of reversible Markov jump processes within time change.
\end{rmk}

\begin{rmk}
  Concerning the third assumption, we cannot prove the result with rate \(f_{i,j}(l_i,l_j)\), but the case where  \(f_{i,j}(l_i,l_j)=f_i(l_i)f_j(l_j)\) can be treated. In fact, by applying a time change, the process with rate function of the form \(f_i(l_i)f_j(l_j)\) can be reduced to our theorem.
\end{rmk}

In section 3, we introduce an equivalent notion of partial exchangeability and, as an example, we give a different proof of partial exchangeability of VRJP within a time scale. Section 4 contains the proof of Theorem~\ref{thm1}.

\section{The two notions of partial exchangeability}
\subsection{Partial exchangeability, infinitesimal point of view}

Consider a nearest neighbor jump process on \(G\) satisfying the third assumption of Theorem~\ref{thm1}. As we have assumed regularity on the trajectory of the process (c.f. Definition~\ref{defi1}), to describe the law of our process, it is enough to describe the probability of the  following events:
\[ \sigma = \{ X_{[0,t_1[}=i_0,X_{[t_1,t_2[}=i_1,X_{[t_2,t_3[}=i_2,\ldots,X_{[t_{n-1},t_n[}=i_{n-1},X_{[t_n,t]}=i_n\},\] which will be denoted
\[\sigma:\ i_0\xrightarrow[]{t_1}i_1\xrightarrow[]{t_2-t_1}i_2\ldots i_{n-1}\xrightarrow[]{t_n-t_{n-1}}i_n\xrightarrow[]{t-t_n}\] in the sequel and we call such an event {\it a trajectory}.

It turns out that when the jump rate is a continuous function of local times, the law of our process can be characterized by some function, which will be called {\it density} in the sequel. In fact, for the study of certain history depending random processes, we have the following lemma:

\begin{lem}
\label{lemm1}
  If \((X_t)\) is a jump process with jump rate depending only on local times and the current position of the random walker, i.e.\ there exists functions \(f_{i,j}(l)\) such that conditioned on the past, \(X_t\) jumps from \(i\) to \(j\) at rate \(f_{i,j}(l(t))\), and, moreover, \(f_{i,j}(l(t))\) does not depend on the variable \(l_i(t)\). Then there exists functions \(d_{\sigma}\), such that for all bounded measurable functions \(\Phi\) defined on the trajectories,
\[\E(\Phi(X_u,\ u\leq t))=\sum_{n\geq 1}\sum_{i_0,\ldots,i_n}\int d_{\sigma}\Phi(\sigma)dt + d_{i_0\xrightarrow[]{t}}\Phi(i_0\xrightarrow[]{t})\]
where
\(d_{\sigma}=\exp(-\int_0^t \sum_{j\sim X_s}f_{X_s,j}(l(s))ds)\prod_{k=1}^nf_{i_{k-1},i_k}(l(t_{k}))\) and
\(d_{i_0\xrightarrow[]{t}}=\P(X_s=i_0, 0\leq s\leq t)\).
\end{lem}
\begin{rmk}
  We believe that Lemma~\ref{lemm1} still hold when \(f_{i,j}(l)\) depends on \(l_i(t)\). In fact, if we can find a time changed process such that its jump rates do not depend on \(l_i(t)\), it is immediate by re-applying the inverse time change that Lemma~\ref{lemm1} holds in the general cases.
\end{rmk}
\begin{proof}
    As \(f_{i,j}(l(t))\) does not depend on \(l_i(t)\), the holding time of \(X_t\) at \(i\) is exponentially distributed with rate
\[\sum_{j\sim i}f_{i,j}(l(t))\]
and the probability of jumping from \(i\) to \(j\) is
\[p(i,j):=\frac{f_{i,j}(l(t))}{\sum_{k\sim i}f_{i,k}(l(t))}.\]
Moreover, the process up to time \(t\) is characterized by the events
\[i_0\xrightarrow[]{s_1}i_1\xrightarrow[]{s_2}\cdots\xrightarrow[]{s_n}i_n\xrightarrow[]{s_{n+1}},\ s_1,\ldots,s_{n+1}>0, \sum_{i=1}^{n+1} s_i\leq t.\]
For \(1\leq k\leq n+1\), denote \(t_k=s_1+\cdots+s_k\),
\begin{align*}
&\P(X_t\text{ follows the trajectory } i_0\xrightarrow[]{s_1}i_1\xrightarrow[]{s_2}\cdots\xrightarrow[]{s_n}i_n\xrightarrow[]{s_{n+1}},\ s_k>0,\ \sum s_k\leq t)\\
&=\int_{t_n\leq t}\prod_{k=1}^n \left(p(i_{k-1},i_k)\exp({-\sum_{j\sim i_{k-1}}f_{i_{k-1},j}(l(t_{k-1}))s_k})\cdot\sum_{j\sim i_{k-1}}f_{i_{k-1},j}(l(t_{k-1}))\right) \P(s_{n+1}>t-t_n) ds \\
&=\int_{t_1<t_2<\cdots<t_n<t}\exp(-\int_0^t\sum_{j\sim X_s}f_{X_s,j}(l(s))ds)\prod_{k=1}^{n}f_{i_{k-1},i_k}(l(t_{k-1}))dt,
\end{align*}with \(ds=ds_1\cdots ds_n,\ dt=dt_1\cdots dt_n\). Now the lemma follows by distinguishing different trajectories.
\end{proof}

\begin{defi}
We say that \(X_t\) admits a density if the assumptions in Lemma~\ref{lemm1} are satisfied, and we denote its density as \(d_{\sigma}\).
\end{defi}

Let us now give another definition of partial exchangeability for continuous time processes in terms of density. Define two trajectories \(\sigma\) and \(\tau\) to be equivalent and denoted \(\sigma \sim \tau\), if their discrete chain strings are equivalent and the local times are equal at each vertex.
Formally,
\begin{defi}
  Let \[\sigma =i_0\xrightarrow[]{t_1}i_1\xrightarrow[]{t_2-t_1}i_2\cdots i_{n-1}\xrightarrow[]{t_n-t_{n-1}}i_n\xrightarrow[]{t-t_n},\]
  \[\tau =j_0\xrightarrow[]{s_1}j_1\xrightarrow[]{s_2-s_1}j_2\cdots j_{n-1}\xrightarrow[]{s_n-s_{n-1}}j_n\xrightarrow[]{t-s_n}.\] Then \(\sigma\) and \(\tau\) are equivalent if and only if
  \[\begin{cases}
    \forall i\in V, \ l_i^{\sigma}(t)=l_i^{\tau}(t)\\
    \forall i,j\ N_{i,j}(\sigma)=N_{i,j}(\tau).
  \end{cases}\]
where \(N_{i,j}(\sigma)\) denotes the number of jumps from \(i\) to \(j\) in \(\sigma\), i.e. \(N_{i,j}(\sigma)=N_{i,j}((i_0,\ldots,i_n))\), and \(l_i^{\sigma}(t)=\int_0^t \mathds{1}_{\sigma_s=i}ds\) denotes the local time.
\end{defi}

\begin{defi}
\label{defdensity}
  A continuous time nearest neighbor jump process is said to be partially exchangeable in density if the densities are equal for any two equivalent trajectories. Or equivalently, the density depends only on final local times and the transition counts.
\end{defi}

\subsection{Equivalence of the two notions}
It turns out that in the case of nearest neighbor jump process with continuous jump rate functions, the notion of partial exchangeability in Definition~\ref{defFreedman} and in Definition~\ref{defdensity} are equivalent.

\begin{prop}
\label{prop1}
  If a continuous time nearest neighbor jump process is partially exchangeable in the sense of Definition~\ref{defdensity}, then it is partially exchangeable in the sense of Definition~\ref{defFreedman}.
\end{prop}

\begin{proof}
  Suppose that the process \(X_t\) is partially exchangeable in density, let \(h>0\), consider the event \(I= \{X_0=i_0,X_h=i_1,\ldots,X_{nh}=i_n\}\), let \((j_0=i_0,j_1,\ldots,j_n)\) be an equivalent string of \((i_0,\ldots,i_n)\), and \(J=\{X_0=j_0,X_h=j_1,\ldots,X_{nh}=j_n\}\).

We construct a bijection \(T\) which maps trajectories of \(I\) to those of \(J\). As \((i_0,\ldots,i_n)\), \((j_0,\ldots,j_n)\) are equivalent, for any pair of neighbors \((i,j)\), there are exactly a same number of transition counts from \(i\) to \(j\). Let us define \(T\) to be the transformation which is a permutation of the time segmentations \([lh,(l+1)h)\) of size \(h\); which, for any \(k\), moves the \(k\)th transition \(i\xrightarrow[]{{k\text{th}}} j\) of \(I\) to the \(k\)th transition \(i\xrightarrow[]{k\text{th}} j\) of \(J\), and leaving the last time segmentation \([nh,\infty)\) invariant. Figure~\ref{fig:1} illustrates an example of such application.

\begin{figure}[!h]
  \centering
  \includegraphics[width=.7\textwidth]{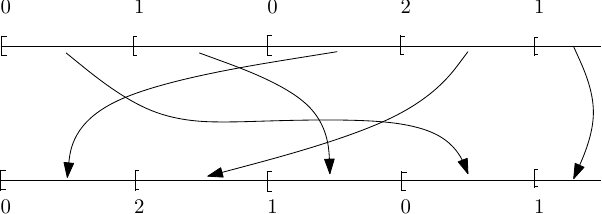}
  \caption{The transformation \(T\) for \(I=\{ X_0=0,X_h=1,X_{2h}=0,X_{3h}=2,X_{4h}=1\} \) and \(J=\{ X_0=0,X_h=2,X_{2h}=1,X_{3h}=0,X_{4h}=1\} \).}
\label{fig:1}
\end{figure}
\noindent Let \[\sigma=k_0\xrightarrow[]{s_1}k_1\xrightarrow[]{s_2}k_2\cdots k_{N-1}\xrightarrow[]{s_N}k_N\xrightarrow[]{s_{N+1}}\] be one trajectory of the event \(I\), we check that \[T(\sigma)=k'_0\xrightarrow[]{s'_1}k'_1\xrightarrow[]{s'_2}k'_2\cdots k'_{N-1}\xrightarrow[]{s'_N}k'_N\xrightarrow[]{s'_{N+1}}\] is a trajectory of the event \(J\), and that \(T\) is one-one and on-to (c.f. Figure~\ref{fig:1.5}).
If we fix the total number of jumps \(N\) and the discrete trajectory \((k_0,k_1,\ldots,k_N)\), then \(T\) can be though as a substitution of integration. Thus

\begin{figure}[!h]
  \centering
  \includegraphics[width=.7\textwidth]{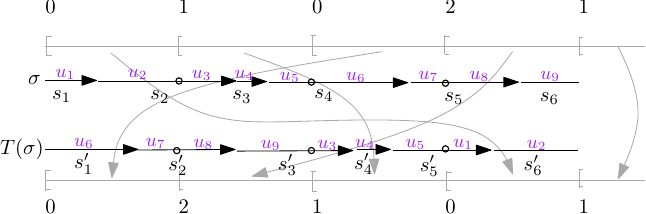}
  \caption{An example of \(\sigma\) and \(T(\sigma)\).}
\label{fig:1.5}
\end{figure}
\begin{align*}
 \P(I)&=\sum_N\sum_{k_0,k_1,\ldots,k_N}\int \mathds{1}_{s_1,\ldots,s_{N+1}\in I(N,k_0,\ldots,k_N)}d_{\sigma}ds_1\cdots ds_{N+1} \\
&=\sum_N\sum_{k'_0,k'_1,\ldots,k'_N}\int \mathds{1}_{s'_1,\ldots,s'_{N+1}\in I'(N,k'_0,\ldots,k'_N)}d_{T(\sigma)} ds'_1\cdots ds'_{N+1} =\P(J),
\end{align*}
where \(I(N,k_0,\ldots,k_N)\) is the subset of \(\mathbb{R}^{N+1}\) defined as the set of \((s_1,\ldots,s_{N+1})\) such that the event \(k_0\xrightarrow[]{s_1}k_1\xrightarrow[]{s_2}\cdots k_{N}\xrightarrow[]{s_{N+1}}\) is in \( I\); and \(I'(N,k'_0,\ldots,k'_N)\) is its image by applying \(T\); see Figure 2 for a concrete example. As \(T\) preserves local times and the numbers of transition counts, these two integrals are whence equal.
\end{proof}

\begin{prop}
  If a jump process is partially exchangeable in the sense of Definition~\ref{defFreedman}, and its jump rate is a continuous function of local times, then it is also partially exchangeable in the sense of Definition~\ref{defdensity}.
\end{prop}

\begin{proof}
  Let \(X_t\) denote such a process, for \(h>0\), consider the \(\sigma\)-algebra \(\mathcal{F}_h=\sigma(X_{nh}, n\geq 0)\), let \[\mathcal{F}_0=\sigma(\cup_{h>0}\mathcal{F}_h)\] and \[\mathcal{F}=\sigma(X_t,t\geq 0).\]
As in~\cite{freedman1996finetti}, we only consider \(h\) running through the binary rationals. Note that \(\mathcal{F}_0=\mathcal{F}\) thanks to the right continuity of the trajectories.

Let \(\sigma=i_0\xrightarrow[]{t_1}i_1\xrightarrow[]{t_2-t_1}i_2\cdots i_n\xrightarrow[]{t-t_n}\) be a trajectory with \(n\) jumps (say \(n\geq 1\) to avoid triviality). Let \(\{X^{(h)}\sim \sigma/h\}\) denotes the event
\[\{ X_0=\sigma_0, \ X_h=\sigma_h,\ldots,X_{Nh}=\sigma_{Nh},\ \text{ with }N=\lfloor t/h\rfloor\}.\]
It turns out that
\[d_{\sigma}=\lim_{h\rightarrow 0}\P(X^{(h)}\sim \sigma/h)h^{-n}.\]
In fact, let \(\Psi=\mathds{1}_{X^{(h)}\sim \sigma/h}\), by definition of \(d_{\sigma}\),
\begin{equation}\label{den}
\E(\Psi( X_u,\ u\leq t))=\P(X^{(h)}\sim \sigma/h)=\sum_{k\geq 1}\sum_{i_1,\ldots,i_k}\int d_{\tau} \Psi(\tau)dt_1\cdots dt_k\end{equation} where  \[\tau =i_0\xrightarrow[]{t_1}i_1\xrightarrow[]{t_2-t_1}i_2\cdots i_{k-1}\xrightarrow[]{t_k-t_{k-1}}i_k\xrightarrow[]{t-t_k}.\]
When \(h\) is small enough, the sum in~\eqref{den} must be over \(k\geq n\), and we have
\begin{align*}
  \P(X^{(h)}\sim \sigma/h)=\P_1+\P_2.
\end{align*} where for some \(p_k,k=1,\ldots, n\) depending on \(h\)
\begin{align*}
\P_1=\P(&(X_u)_{0\leq u\leq t} \text{ makes }n \text{ jumps at times }s_1,\ldots,s_n \\
 &\text{ with }s_k\in (p_kh,(p_k+1)h] \text{ and the trajectory is }i_0,\ldots,i_n)\\
\P_2=\P(&(X_u)_{0\leq u\leq t} \text{ makes more than }n+1 \text{ jumps and } X^{(h)}\sim \sigma/h)
\end{align*}
Note that the jump rates are bounded from both below and above, and any holding time in the event of \(\P_2\) must be in an interval of length lesser than \(2h\), whence the probability of making \(n+l\) (\(l\geq 1\)) jumps following the trajectory \(\sigma/h\) is smaller than the probability of \(n+l\) independent exponential variables (of parameter \(C\)) each smaller than \(2h\), where \(C\) is an upper bound of the jump rates. Whence
\begin{align*}
  \P_2 \leq \sum_{l\geq 1} (\P(\text{cst}\leq \mathcal{E}xp(C)<\text{cst}+2h))^{n+l}\leq  \sum_{l\geq 1} (\P(\mathcal{E}xp(C)<2h))^{n+l} =O(h^{n+1}) .
\end{align*}
Thus \(\P_2\) can be dropped when taking the limit. In addition,
\[\P_1=\int_{p_nh}^{(p_n+1)h}\cdots \int_{p_1h}^{(p_1+1)h} d_{\sigma}\ dt_1\cdots dt_n,\] note that here \(d_{\sigma}\) depends only on \(t_1,\ldots,t_n\) and it is an absolutely integrable function, by Lebesgue differentiation theorem (Theorem 1.6.19~\cite{tao2011introduction}) \(\lim_{h\rightarrow 0}\P_1/h^{n}=d_{\sigma}\).
\begin{figure}[!h]
  \centering
  \includegraphics[width=.5\textwidth]{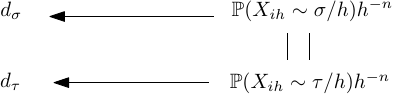}
\end{figure}
Now let \(\sigma\sim \tau\), when \(h\) is sufficiently small, proceeding as in the diagram shows that \(d_{\sigma}=d_{\tau}\).
\end{proof}

\subsection{Example: VRJP is partially exchangeable within a time change}

Recall that \(Y_s=X_{D^{-1}(s)}\) with \(D(s)=\sum_{i\in V}(l_i(s)^2+2l_i(s))\), It turns out that we can write down the density of the trajectory \(\sigma\) of the (time changed) VRJP process \(Y\) (For convenience, write \(s_{n+1}\) for \(s\) in the sequel). The density of
\[\sigma:= \ i_0\xrightarrow[]{s_1}i_1\xrightarrow[]{s_2-s_1}i_2\cdots i_{n-1}\xrightarrow[]{s_n-s_{n-1}}i_n\xrightarrow[]{s-s_n}\]
is (c.f.~\cite{sabot2013ray}), denoting \(S_i(t)=\int_0^t\mathds{1}_{Y_u=i}du\) the local time of \(Y\),
\begin{equation}
\label{densityST}
\begin{aligned}
d_{\sigma}=&(\frac{1}{2})^n\prod_{k=1}^nW_{i_{k-1},i_k}\prod_{i\in V, i\neq i_n}\frac{1}{\sqrt{1+S_i(s)}}\\
&\cdot\exp{(- \sum_{i \sim j} W_{i,j}(\sqrt{(S_i(s)+1)(S_j(s)+1)}-1))},
\end{aligned}
\end{equation}
 which clearly depends only on final local times and transition counts, thus by Proposition~\ref{prop1}, \(Y\) is partially exchangeable. On finite graph it is rather easy to prove that the VRJP is recurrent (for example, using a representation of VRJP by time changed Poisson point process as in~\cite{sabot2011edge}, and then use an argument as in~\cite{davis1990reinforced} or~\cite{sellke1994reinforced}). Therefore, \(Y\) is a mixture of Markov jump processes.

For convenient, we include a proof of this in the sequel (after the proof of Proposition~\ref{prop3}), since the mechanisms of the proof enlightens the proof of the main theorem.

\section{Proof of Theorem~\ref{thm1}}
\subsection{Computation of densities}
Let \(X\) be a nearest neighbor jump process on \(G\) satisfying the assumptions of Theorem~\ref{thm1}, in particular, recall the time scale
\begin{equation} \label{timescale} D(s)=\sum_{i\in V}h_i(l_i(s)).\end{equation}
Let \(l_i(t)\) be the local time of the process \(X\) at vertex \(i\) at time \(t\). Let us denote the process after time change to be \begin{equation}\label{processY} Y_t=X_{D^{-1}(t)},\end{equation} let \begin{equation} \label{localtimeY}S_i(s)=\int_0^s \mathds{1}_{Y_u=i}du\end{equation} denote the local time of \(Y\). Consider the trajectory
\begin{equation}\label{trajY}\sigma: \ i_0\xrightarrow[]{t_1}i_1\xrightarrow[]{t_2-t_1}i_2\cdots i_{n-1}\xrightarrow[]{t_n-t_{n-1}}i_n\xrightarrow[]{t-t_n}\end{equation} where \(0<t_1<\cdots <t_n<t\), after applying the time change, the corresponding trajectory for \(Y\) is
 \[\sigma_Y: \ i_0\xrightarrow[]{s_1}i_1\xrightarrow[]{s_2-s_1}i_2\cdots i_{n-1}\xrightarrow[]{s_n-s_{n-1}}i_n\xrightarrow[]{s-s_n}\] where \(s_k=D(t_k)\).

\begin{prop}
\label{prop3}
  With the same settings as in equations~\eqref{timescale}~\eqref{processY}~\eqref{localtimeY}~\eqref{trajY}, the density of the trajectory \(\sigma_Y\) for \(Y\) is
\[d_{\sigma}^Y=\exp{\left(-\int_0^{s}\sum_{j\sim Y_v}\frac{f_{Y_v,j}(h_j^{-1}(S_j(v)))}{h'_{Y_v}(h_{Y_v}^{-1}(S_{Y_v}(v)))}dv\right)}\prod_{k=1}^n \frac{f_{i_{k-1},i_k}(h_{i_k}^{-1}(S_{i_k}(s_{k-1})))}{h'_{i_{k-1}}(h^{-1}_{i_{k-1}}(S_{i_{k-1}}(s_k)))}.\]
\end{prop}
\begin{proof}
Applying Lemma~\ref{lemm1} to the process \(X\),
\[d_{\sigma}=\exp{\left(-\int_0^{t}\sum_{j\sim X_u} f_{X_u,j}(l_j(u))du\right)}\prod_{k=1}^n f_{i_{k-1},i_k}(l_{i_k}(t_{k-1})).\]
Recall that in~\eqref{timescale} we assumed that \(h_i: \mathbb{R^+}\rightarrow \mathbb{R^+}\) are diffeomorphisms satisfying \(h_i(0)=0\).

Next we compute the density for the same trajectory \(\sigma\) but for the process \(Y_s=X_{D^{-1}(s)}\), as we have \(S_i(D(s))=h_i(l_i(s))\), derivation leads to
\[S_i(D(s))'=D'(s)\mathds{1}_{Y_{D(s)}=i}=h_i'(l_i(s))\mathds{1}_{X_s=i}.\]
Hence \[(D^{-1}(t))'=\frac{1}{D'(D^{-1}(t))}=\frac{1}{h'_{Y_{t^-}}\circ  h_{Y_{t^-}}^{-1}(S_{Y_{t^-}}(t))},\]
\[l_{i_k}(t_{k-1})=h_{i_k}^{-1}(S_{i_k}(D(t_{k-1})))=h^{-1}_{Y_{s_k}}(S_{Y_{s_k}}(s_{k-1})).\]
Substituting \(s=D(t)\), we have
\begin{align*}d^Y_{\sigma}=
\exp{(-\int_0^{s}\sum_{j\sim Y_v}\frac{f_{Y_v,j}(h_j^{-1}(S_j(v)))}{h'_{Y_v}(h_{Y_v}^{-1}(S_{Y_v}(v)))}dv)}\prod_{k=1}^n \frac{f_{i_{k-1},i_k}(h_{i_k}^{-1}(S_{i_k}(s_{k-1})))}{h'_{i_{k-1}}(h^{-1}_{i_{k-1}}(S_{i_{k-1}}(s_k)))}.
\end{align*}
\end{proof}

\noindent{\bf Back to the partial exchangeability of VRJP}

\begin{proof}
Apply the previous proposition to VRJP, where \(f_{i,j}(l_j)=W_{i,j}(1+l_j)\) and \(h_i(l_i)=l_i^2+2l_i\).
The density \(d_{\sigma}^Y\) is
\[\frac{1}{2^n}\exp\left ( -\int_0^{s}\sum_{j\sim Y_u} \frac{W_{Y_u,j}\sqrt{S_j(u)+1}}{2\sqrt{S_{Y_u}(u)+1}}du\right) \prod_{k=1}^n\left(W_{i_{k-1},i_k} \frac{\sqrt{S_{i_k}(s_{k-1})+1}}{\sqrt{S_{i_{k-1}}(s_k)+1}}\right).\]
As our trajectory is left continuous without explosion, starting at \(i_0\), if we calculate the product through the trajectory, by telescopic simplification, it results that the product reduces to
\[\prod_{i\in V\ i\neq i_n}\frac{1}{\sqrt{S_i(s)+1}} \prod_{k=1}^n W_{i_{k-1},i_k}.\]
To compute the integral inside the exponential, it is enough to note that, in the expression:
\[\sum_{i\sim j}W_{i,j}(\sqrt{(S_i(s)+1)(S_j(s)+1)}-1),\]
the local times \(S_i(s),i\in V\) of the process \(Y\) only vary (linearly) with \(s\) when the process is at \(i\), i.e., when \(Y_t=i\). Therefore, the derivative of the above expression with respect to \(s\) equals to
\[\sum_{j\sim Y_s}\frac{W_{Y_s,j}\sqrt{S_j(s)+1}}{2\sqrt{S_{Y_s}(s)+1}}\]
which is what we integrate inside the exponential.

Whence~\eqref{densityST} is proved, and expression~\eqref{densityST} depends only on final local times and transition counts, the result hence follows.
\end{proof}

\subsection{Determination of time change \(h\)}
In the sequel we work with the time changed process \(Y\), to simplify notations, we will write \(d_{\sigma}\) for \(d^Y_{\sigma}\) when it does not lead to any confusion. By Proposition~\ref{prop3}, the density of certain trajectory contains an exponential term and a product term, let us denote
\[d_{\sigma}=\exp(-\int\sigma) \cdot \prod \sigma,\]
with \[\begin{cases}
           \int \sigma=\int_0^{s}\sum_{j\sim Y_v}\frac{f_{Y_v,j}(h_j^{-1}(S_j(v)))}{h'_{Y_v}(h_{Y_v}^{-1}(S_{Y_v}(v)))}dv\\
           \prod \sigma =\prod_{k=1}^n \frac{f_{i_{k-1},i_k}(h_{i_k}^{-1}(S_{Y_{s_k}}(s_{k-1})))}{h'_{i_{k-1}}(h^{-1}_{i_{k-1}}(S_{Y_{s_{k-1}}}(s_k)))}
        \end{cases}\]
where the exponential term stems from those exponential waiting times, and the product term corresponds to  the probability of the discrete chain.

The heuristics of the proof in this subsection is the following: as we assumed partial exchangeability, if we consider two equivalent trajectories, then their densities  share the same expression, by comparing them we can hence deduce certain equalities involving \(f_{i,j}\) and \(h_i\) etc. It turns out that these equalities determine \(h_i\)s then \(f_{i,j}\)s.

The following fact is simple but important, suppose that at time \(s\), the random walker arrives at \(i_0\), each vertex \(i\) has accumulated local time \(l_i:=S_i(s)\); then it jumps to \(i_1\) after an amount of time \(t\), by Proposition~\ref{prop3}, the density has acquired a multiplicative factor
\begin{equation}
\label{factor}
\exp\left(-\int_s^{s+t}\sum_{j\sim i_0}\frac{f_{i_0,j}\circ h_j^{-1}(l_j)}{h'_{i_0}\circ h_{i_0}^{-1}(l_{i_0}+v)}dv\right)\cdot\frac{f_{i_0,i_1}\circ h_{i_1}^{-1}(l_{i_1})}{h'_{i_0}\circ h_{i_0}^{-1}(l_{i_0}+t)}.
\end{equation}
This fact is in constant use in the sequel, when we explicit the density of certain trajectory.
\begin{lem}
\label{lem0}
  Let  \(\sigma=i_0\xrightarrow[]{s_1}i_1\xrightarrow[]{s_2-s_1}i_2\cdots i_{n-1}\xrightarrow[]{s_n-s_{n-1}}i_n\xrightarrow[]{s-s_n} \) be a trajectory, then \(\displaystyle \int \sigma=\int \tilde{\sigma} + \int \hat{\sigma}\) where
\[\int \tilde{\sigma}= \int_0^s \sum_{j\in \sigma, j\sim Y_v}\frac{f_{Y_v,j}(h_j^{-1}(S_j(v)))}{h'_{Y_v}(h_{Y_v}^{-1}(S_{Y_v}(v)))}dv,\ \int\hat{\sigma}=\int_0^s \sum_{j\notin \sigma, j\sim Y_v}\frac{f_{Y_v,j}(h_j^{-1}(S_j(v)))}{h'_{Y_v}(h_{Y_v}^{-1}(S_{Y_v}(v)))}dv \] and if \(\tau\) is such that \(\tau\sim \sigma\), then
\(\displaystyle \int \hat{\sigma}=\int \hat{\tau}.\)
\end{lem}
\begin{proof}
  Note that for \(j\notin \sigma\), \(S_j(u)=0\) for all \(u\leq s\). Let \(\hat{H}_i\) be the primitive of \(\displaystyle \frac{1}{h_i'\circ h_i^{-1}}\) such that \(\hat{H}_i(0)=0\),
  \begin{align*}
    \int \hat{\sigma}&=\sum_{j\notin \sigma}\int_0^s \mathds{1}_{Y_v\sim j} \frac{f_{Y_v,j}(0)}{h'_{Y_v}(h_{Y_v}^{-1}(S_{Y_v}(v)))}dv \\
&=\sum_{j\notin \sigma,i\in \sigma,j\sim i}f_{i,j}(0) \int_0^s \frac{\mathds{1}_{Y_v=i}}{h'_i(h_i^{-1}(S_i(v)))}dv \\
&=\sum_{j\notin \sigma,i\in \sigma,j\sim i}f_{i,j}(0) \hat{H}_i(S_i(s))
  \end{align*} which depends only on final local times, thus if \(\tau\sim \sigma\), then \(\displaystyle \int\hat{\tau}=\int\hat{\sigma}\).
\end{proof}
In the sequel \(cst\) denotes some constant, which can vary from line to line.
\begin{lem}
\label{lem1}
  If the process \(X\) admits such a time change \(D\) which makes it partially exchangeable in density, then for any \(i\sim j\), there exists some constants \(\lambda_{i,j}\) such that
  \begin{equation}
    \label{eq:1}
    f_{i,j}(x)=\lambda_{i,j}h_j'(x),\ \forall x\geq 0.
  \end{equation}

\end{lem}

\begin{proof}
  Let \(\epsilon>0\), consider the following two trajectories for the process \(Y\):
\[ \sigma = i\xrightarrow[]{\epsilon}j\xrightarrow[]{\epsilon}i\xrightarrow[]{t}j\xrightarrow[]{s}i \xrightarrow[]{\cdot}\] 
\[ \tau = i\xrightarrow[]{t}j\xrightarrow[]{s}i\xrightarrow[]{\epsilon}j\xrightarrow[]{\epsilon}i\xrightarrow[]{\cdot}\] 
Note that \(\sigma\) and \(\tau\) have the same transition counts and the final local times on vertex \(i,j\) are respectively equal. Thus the densities of these trajectories are a.s.\ equal by partial exchangeability. By Lemma~\ref{lem0},

\begin{align*}
d_{\sigma}=  \prod \sigma \cdot \exp(\int \tilde{\sigma}+\int\hat{\sigma}),
\end{align*}
where \[ \begin{cases} \prod \sigma =\frac{f_{i,j}\circ h_j^{-1}(0)}{h_i'\circ h_i^{-1}(\epsilon)}\cdot\frac{f_{j,i}\circ h_i^{-1}(\epsilon)}{h_j'\circ h_j^{-1}(\epsilon)}\cdot
\frac{f_{i,j}\circ h_j^{-1}(\epsilon)}{h_i'\circ h_i^{-1}(\epsilon+t)}\cdot\frac{f_{j,i}\circ h_i^{-1}(\epsilon+t)}{h_j'\circ h_j^{-1}(\epsilon+s)}\\
\int \tilde{\sigma}=
  \int_0^{\epsilon}\frac{f_{i,j}\circ h_j^{-1}(0)}{h_i'\circ h_i^{-1}(v)}dv+
 \int_0^{\epsilon}\frac{f_{j,i}\circ h_i^{-1}(\epsilon)}{h_j'\circ h_j^{-1}(v)}dv+
 \int_0^{t}\frac{f_{i,j}\circ h_j^{-1}(\epsilon)}{h_i'\circ h_i^{-1}(\epsilon+v)}dv+
  \int_0^{s}\frac{f_{j,i}\circ h_i^{-1}(\epsilon+t)}{h_j'\circ h_j^{-1}(\epsilon+v)}dv. \end{cases}\]

\begin{align*}
d_{\tau}=\prod \tau \cdot \exp(\int \tilde{\tau}+\int\hat{\tau}),
\end{align*}
where \[\begin{cases} \prod \tau=
\frac{f_{i,j}\circ h_j^{-1}(0)}{h_i'\circ h_i^{-1}(t)}\cdot\frac{f_{j,i}\circ h_i^{-1}(t)}{h_j'\circ h_j^{-1}(s)}\cdot
\frac{f_{i,j}\circ h_j^{-1}(s)}{h_i'\circ h_i^{-1}(t+\epsilon)}\cdot\frac{f_{j,i}\circ h_i^{-1}(\epsilon+t)}{h_j'\circ h_j^{-1}(\epsilon+s)} \\
\int \tilde{\tau}=
 \int_0^{t}\frac{f_{i,j}\circ h_j^{-1}(0)}{h_i'\circ h_i^{-1}(v)}dv+
 \int_0^{s}\frac{f_{j,i}\circ h_i^{-1}(t)}{h_j'\circ h_j^{-1}(v)}dv+
 \int_0^{\epsilon}\frac{f_{i,j}\circ h_j^{-1}(s)}{h_i'\circ h_i^{-1}(t+v)}dv+
  \int_0^{\epsilon}\frac{f_{j,i}\circ h_i^{-1}(\epsilon+t)}{h_j'\circ h_j^{-1}(s+v)}dv;\end{cases}\]
We do not explicit \(\int \hat{\sigma}\) and \(\int\hat{\tau}\) as they cancel when we compare these expressions (c.f. Lemma~\ref{lem0}).

Letting \(\epsilon\rightarrow 0\) yields that \(\exp(\int \tilde{\sigma})=\exp(\int \tilde{\tau})\); therefore \(\prod \sigma=\prod \tau\), i.e.
\[\forall s,t,\ \frac{f_{i,j}\circ h_j^{-1}(s)}{h_j'\circ h_j^{-1}(s)} \cdot \frac{f_{j,i} \circ h_i^{-1}(t)}{h_i'\circ h_i^{-1}(t)}=cst.\]
Now fix \(t\), let \(s\) vary, whence \[\forall s,\ f_{i,j}\circ h_j^{-1}(s)=cst\cdot h_j'\circ h_j^{-1}(s),\] and let \(\lambda_{i,j}\) denotes this constant, as \(h_j^{-1}\) is a diffeomorphism, its range is \(\mathbb{R}^+\), which allows us to conclude.
\end{proof}

 The next lemma states in some sense that the exponential part and the product part appearing in the density of a trajectory can be treated separately.
\begin{lem}
\label{lem2}
  Let \(\sigma,\tau\) be two trajectories,  and denote \[d_{\sigma}=\exp(\int \sigma) \cdot \prod \sigma,\hspace{.5cm} d_{\tau}=\exp(\int \tau)\cdot \prod \tau,\] if \(\sigma \sim \tau\), then \(\prod \sigma =\prod \tau\).
\end{lem}

\begin{proof}
  We have \(S_{Y_{s_k}}(s_k)=S_{Y_{s_k}}(s_{k-1})\), thus Lemma~\ref{lem1} yields that \(f_{i_{k-1},i_k}\circ h_{i_k}^{-1}(S_{Y_{s_k}}(s_{k-1}))=\lambda_{i_{k-1},i_k}h_{i_k}'\circ h_{i_k}^{-1}(S_{Y_{s_k}}(s_{k}))\).  Whence the product part is \[\prod \sigma = \prod_{k=1}^n \frac{f_{i_{k-1},i_k}(h_{i_k}^{-1}(S_{Y_{s_k}}(s_{k-1})))}{h'_{i_{k-1}}(h^{-1}_{i_{k-1}}(S_{Y_{s_{k-1}}}(s_k)))}=\prod_{k=1}^n \lambda_{i_{k-1},i_k}\frac{\prod_{i\neq i_0}h_i'\circ h_i^{-1}(0)}{\prod_{i\neq i_n}h_i'\circ h_i^{-1}(S_i(s))},\]  and the last term depends only on the transition counts and final local times.
\end{proof}

\begin{lem}
\label{lem3}
 Let \(H_i=h_i'\circ h_i^{-1}\), then for some constant \(A_i\) (recall that \(h_i\) is assumed \(\mathcal{C}^2\) diffeomorphism), \[ (H_i^2)'=A_i \text{ and if  }i\sim j, \text{ then } \lambda_{i,j}A_j=\lambda_{j,i}A_i.\]
\end{lem}
\begin{rmk}
The latest equality tells that the process is reversible. However, we did not assume the reversibility of the process, but vertex reinforced jump processes are reversible (as a mixture of reversible Markov jump process), so are the edge reinforced random walks. In contrast, directed edge reinforced random walks are mixtures of non reversible Markov chains, with independent Dirichlet environments. We can hence expect that the reversibility is a consequence of a non oriented linear reinforcement (where linearity leads to partial exchangeability).
\end{rmk}
\begin{proof}

Recall that we have assumed that the graph is strongly connected, i.e.\ if \(i,j\) are two adjacent vertices, there exists a shortest cycle \(i_1\sim i_2\sim i_3\cdots \sim i_n\sim i_1\) with \(i_1=i,i_n=j\) and the \(i_k\)s are distinct and \(n\geq 2\).

\begin{figure}[!h]
  \centering
  \includegraphics[width=.7\textwidth]{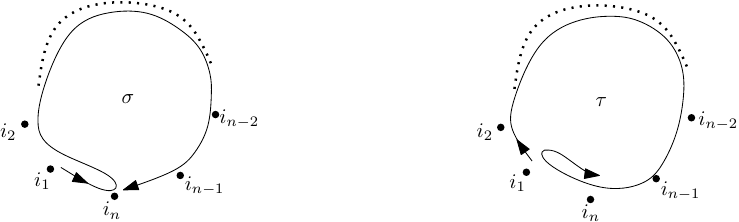}
  \caption{the trajectories \(\sigma\) and \(\tau\) in Lemma~\ref{lem3}.}
\label{fig:st}
\end{figure}

Let \((i_1=i,i_2,i_3,\ldots,i_n=j)\) be a cycle as described, consider the trajectories (c.f. Figure~\ref{fig:st})
 \[\sigma=i_1\xrightarrow[]{r_1}i_n\xrightarrow[]{r_2}i_1\xrightarrow[]{s_1}i_2\xrightarrow[]{s_2}i_{3}\cdots i_{n-2}\xrightarrow[]{s_{n-2}}i_{n-1}\xrightarrow[]{s_{n-1}}i_n \]
\[\tau=i_1\xrightarrow[]{r_1}i_2\xrightarrow[]{s_2}i_{3} \cdots i_{n-2}\xrightarrow[]{s_{n-2}}i_{n-1}\xrightarrow[]{s_{n-1}}i_n\xrightarrow[]{r_2}i_1\xrightarrow[]{s_1}i_n.\]
As \(\sigma\sim \tau\), by Lemma~\ref{lem2} and Lemma~\ref{lem0}, \(\int \tilde{\sigma}=\int \tilde{\tau}\). Also let
\[\sigma'=i_1\xrightarrow[]{r_1}i_n\xrightarrow[]{r_2}i_1\xrightarrow[]{s_1}i_2\xrightarrow[]{s_2}i_1\]
\[\tau'=i_1\xrightarrow[]{r_1}i_2\xrightarrow[]{s_2}i_1\xrightarrow[]{s_1}i_n\xrightarrow[]{r_2}i_1,\] thus \(\int \tilde{\sigma'}=\int \tilde{\tau'}\).
We are going to compute explicitly \(\int \tilde{\sigma},\ \int\tilde{\tau}\) etc, using~\eqref{factor}, let \(s=r_1+r_2+s_1+\cdots+s_{n-1}\) and recall that \(\hat{H}_i\) is the primitive of \(\displaystyle \frac{1}{h_i'\circ h_i^{-1}}\) such that \(\hat{H}_i(0)=0\).
\begin{align*}
\int\tilde{\sigma}=&\sum_{(i,j)\in\sigma^2,i\sim j} \lambda_{i,j}\int_0^{s} \mathds{1}_{Y_v=i}\frac{h_j'\circ h_j^{-1}(S_j(v))}{h_i'\circ h_i^{-1}(S_i(v))}dv \\
&=\lambda_{i_1,i_2}H_{i_2}(0)\hat{H}_{i_1}(r_1+s_1)  + \lambda_{i_2,i_1}H_{i_1}(r_1+s_1)\hat{H}_{i_2}(s_2)\\
&+\lambda_{i_1,i_n}\left(H_{i_n}(0)\hat{H}_{i_1}(r_1)+H_{i_n}(r_2)(\hat{H}_{i_1}(r_1+s_1)-\hat{H}_{i_1}(r_1))\right)\\
&+\lambda_{i_n,i_1}H_{i_1}(r_1)\hat{H}_{i_n}(r_2)+\lambda_{i_n,i_{n-1}}H_{i_{n-1}}(0)\hat{H}_{i_n}(r_2)\\
&+\lambda_{i_{n-1},i_n}H_{i_n}(r_2)\hat{H}_{i_{n-1}}(s_{n-1}) + \Delta
\end{align*}
where \(\Delta\) is defined as follows: let \(Q_k:=H_{i_{k}}(0)\hat{H}_{i_{k-1}}(s_{i_{k-1}})\) and \(Q_k':=H_{i_k}(s_k)\hat{H}_{i_{k+1}}(s_{i_{k+1}})\),
\begin{align*}
  \Delta=\sum_{k=3}^{n-1}\lambda_{i_{k-1},i_k}Q_k+\lambda_{i_k,i_{k-1}}Q'_{k-1}.
\end{align*}
For \(\tilde{\tau}\) we have:
\begin{align*}
\int\tilde{\tau}=&\sum_{(i,j)\in\tau^2,i\sim j} \lambda_{i,j}\int_0^{s} \mathds{1}_{Y_v=i}\frac{h_j'\circ h_j^{-1}(S_j(v))}{h_i'\circ h_i^{-1}(S_i(v))}dv \\
&=\lambda_{i_1,i_2}H_{i_2}(0)\hat{H}_{i_1}(r_1)+H_{i_2}(s_2)(\hat{H}_{i_1}(r_1+s_1)-\hat{H}_{i_1}(r_1)) \\
&+\lambda_{i_2,i_1}H_{i_1}(r_1)\hat{H}_{i_2}(s_2)\\
&+\lambda_{i_1,i_n}\left(H_{i_n}(0)\hat{H}_{i_1}(r_1)+H_{i_n}(r_2)(\hat{H}_{i_1}(r_1+s_1)-\hat{H}_{i_1}(r_1))\right)\\
&+\lambda_{i_n,i_1}H_{i_1}(r_1)\hat{H}_{i_n}(r_2)+\lambda_{i_n,i_{n-1}}H_{i_{n-1}}(s_{n-1})\hat{H}_{i_n}(r_2)\\
&+\lambda_{i_{n-1},i_n}H_{i_n}(0)\hat{H}_{i_{n-1}}(s_{n-1}) + \Delta
\end{align*} with the same \(\Delta\). Also
\begin{align*}
\int\tilde{\sigma}'&=\lambda_{i_1,i_2}\left(H_{i_2}(0)\hat{H}_{i_1}(r_1)+H_{i_2}(0)(\hat{H}_{i_1}(r_1+s_1)-\hat{H}_{i_1}(r_1))\right) \\
&+\lambda_{i_2,i_1}H_{i_1}(r_1+s_1)\hat{H}_{i_2}(s_2)\\
&+\lambda_{i_1,i_n}\left(H_{i_n}(0)\hat{H}_{i_1}(r_1)+H_{i_n}(r_2)(\hat{H}_{i_1}(r_1+s_1)-\hat{H}_{i_1}(r_1))\right)\\
&+\lambda_{i_n,i_1}H_{i_1}(r_1)\hat{H}_{i_n}(r_2)
\end{align*}
\begin{align*}
\int\tilde{\tau}'&=\lambda_{i_1,i_2}\left(H_{i_2}(0)\hat{H}_{i_1}(r_1)+H_{i_2}(s_2)(\hat{H}_{i_1}(r_1+s_1)-\hat{H}_{i_1}(r_1))\right) \\
&+\lambda_{i_2,i_1}H_{i_1}(r_1)\hat{H}_{i_2}(s_2)\\
&+\lambda_{i_1,i_n}\left(H_{i_n}(0)\hat{H}_{i_1}(r_1)+H_{i_n}(0)(\hat{H}_{i_1}(r_1+s_1)-\hat{H}_{i_1}(r_1))\right)\\
&+\lambda_{i_n,i_1}H_{i_1}(r_1+s_1)\hat{H}_{i_n}(r_2).
\end{align*}
Recall that \(\int\tilde{\sigma}-\int\tilde{\sigma}'=\int\tilde{\tau}-\int\tilde{\tau}'\), which leads to
\begin{align*}
  &\lambda_{i_n,i_{n-1}}H_{i_{n-1}}(0)\hat{H}_{i_n}(r_2)+\lambda_{i_{n-1},i_n}H_{i_n}(r_2)\hat{H}_{i_{n-1}}(s_{n-1})\\
&=\lambda_{i_1,i_n}(H_{i_n}(r_2)-H_{i_n}(0))(\hat{H}_{i_1}(r_1+s_1)-\hat{H}_{i_1}(r_1))\\
&+\lambda_{i_n,i_1}(H_{i_1}(r_1)-H_{i_1}(r_1+s_1))\hat{H}_{i_n}(r_2)\\
 &+\lambda_{i_n,i_{n-1}}H_{i_{n-1}}(s_{n-1})\hat{H}_{i_n}(r_2)+\lambda_{i_{n-1},i_n}H_{i_n}(0)\hat{H}_{i_{n-1}}(s_{n-1})
\end{align*}
letting \(s_{n-1}\rightarrow 0\) leads to
 \begin{align*}
  & \lambda_{i_1,i_n}(H_{i_n}(r_2)-H_{i_n}(0))(\hat{H}_{i_1}(r_1+s_1)-\hat{H}_{i_1}(r_1))=\\
  &\lambda_{i_n,i_1}(H_{i_1}(r_1+s_1)-H_{i_1}(r_1))\hat{H}_{i_n}(r_2)
 \end{align*} as \(i_1,i_n,r_2,s_1,r_1\) are arbitrary, divide the formula by \(r_2s_1\) and let \(r_2,s_1\) go to zero leads to
\[\lambda_{i_1,i_n}H'_{i_n}(0)\hat{H}'_{i_1}(r_1)=\lambda_{i_n,i_1}H'_{i_1}(r_1)\hat{H}'_{i_n}(0),\] finally note that \(\hat{H}'_i=1/H_i\), thus \(\lambda_{i_1,i_n}(H_{i_n}^2)'(0)=\lambda_{i_n,i_1}(H_{i_1}^2)'(r_1)\).
\end{proof}

\begin{lem}
  For all \(i\sim j\), let \(W_{i,j}=\lambda_{i,j}A_j/2=\lambda_{j,i}A_i/2\), there exists constant \(\varphi_j\) depends only on \(j\), such that \(f_{i,j}(x)=W_{i,j}x+\varphi_j\).
\end{lem}

\begin{proof}
As \((H_j^2(s))'=A_j\), there exists \(B_j\) such that \(H_j^2(s)=A_j s + B_j\), therefore
\[f_{i,j}\circ h_j^{-1}(s)=\lambda_{i,j}H_j(s)=\lambda_{i,j}\sqrt{A_js+B_j}.\]
On the other hand, \((h^{-1}_j)'(s)=\frac{1}{\sqrt{A_js+B_j}}\), thus for some \(C_j\),
\[h_j^{-1}(s)=\frac{2}{A_j}\sqrt{A_js+B_j}+C_j.\]
\(f_{i,j}(h_j^{-1}(s))=f_{i,j}(\frac{2}{A_j}\sqrt{A_js+B_j}+C_j)=\lambda_{i,j}\sqrt{A_js+B_j}\), which leads to
\[f_{i,j}(x)=W_{i,j}x+\varphi_j,\] where \(\varphi_j\) is some constant depends only on \(j\).
Applying the time change
\[\tilde{D}(s)=\sum_i \frac{l_i(s)-\varphi_i}{\varphi_i},\] the resulting process will be of jump rate \[W_{i,j}\varphi_i\varphi_j(1+T_j(t))\] where \(T_j(t)\) is the local time for the time changed process \(Z_t=X_{\tilde{D}^{-1}(t)}\).
\end{proof}

\noindent{\bf Acknowledgments:}
I would like to thank Christophe Sabot for his constant support in this project. I would like to thank an anonymous referee for carefully reading the paper and providing corrections.

\nocite{*}
\bibliography{bibi}{}
\bibliographystyle{plain}
\end{document}